\documentclass[final]{amsart}

\usepackage{enumerate}
\usepackage{amssymb,amsfonts,amsmath,amsthm}
\usepackage{tikz}
\usepackage{showkeys}
\usepackage[all]{xy}

\tikzset{node distance=2cm, auto}

\theoremstyle{remark}
\newtheorem{example}{Example}[section]
\newtheorem{remark}[example]{Remark}
\theoremstyle{definition}
\newtheorem{definition}[example]{Definition}
\theoremstyle{plain}
\newtheorem{proposition}[example]{Proposition}
\newtheorem{corollary}[example]{Corollary}

\newtheorem{theorem}[example]{Theorem}
\newtheorem{lemma}[example]{Lemma}

\newcommand{\mref}[1]{(\ref{#1})}

\newcommand{\ZZ}{\mathbb{Z}}

\newcommand{\sB}{\mathcal{B}}
\newcommand{\sL}{\mathcal{L}}
\newcommand{\sM}{\mathcal{M}}

\newcommand{\sE}{\mathcal{E}}
\newcommand{\sG}{\mathcal{G}}
\newcommand{\sO}{\mathcal{O}}
\newcommand{\sT}{{\mathcal{T}}}
\newcommand{\sF}{{\mathcal{F}}}
\newcommand{\sW}{{\mathcal{W}}}
\newcommand{\sQ}{{\mathcal{Q}}}
\newcommand{\sV}{{\mathcal{V}}}

\newcommand{\bD}{{\bf D}}
\newcommand{\bR}{{\bf R}}

\newcommand{\flag}{\mathrm{Flag}}
\newcommand{\gr}{\mathrm{Gr}}

\DeclareMathOperator{\Spec}{Spec}

\DeclareMathOperator{\rD}{D}

\DeclareMathOperator{\Qcoh}{\textbf{Qcoh}}

\DeclareMathOperator{\Aut}{Aut}
\DeclareMathOperator{\End}{End}
\DeclareMathOperator{\Hom}{Hom}
\DeclareMathOperator{\Ext}{Ext}

\DeclareMathOperator{\univ}{univ}
\DeclareMathOperator{\spec}{Spec}
\DeclareMathOperator{\Isom}{Isom}
\DeclareMathOperator{\SB}{SB}
\DeclareMathOperator{\Pic}{Pic}

\DeclareMathOperator{\Gal}{Gal}
\DeclareMathOperator{\pgl}{PGL}
\DeclareMathOperator{\gl}{GL}

\newcommand{\CC}{{\mathbb C}}

\begin{document}

\title{Pushforwards of tilting sheaves}

\author[A. Dhillon]{Ajneet Dhillon}
\address{Department of Mathematics, Western University, Canada}
\email{adhill3@uwo.ca}

\author[N. Lemire]{Nicole Lemire}
\address{Department of Mathematics, Western University, Canada}
\email{nlemire@uwo.ca}

\author[Y. Yan]{Youlong Yan}
\address{Department of Mathematics, Western University, Canada}
\email{yyoulong@gmail.com}

\begin{abstract}
We investigate the behaviour of tilting sheaves under pushforward by a  finite Galois
morphism.
We determine conditions under which such a pushforward of a tilting sheaf is 
a tilting sheaf.  We then produce some examples of Severi-Brauer flag varieties
and arithmetic toric varieties in which our method produces a tilting sheaf,
adding to the list of positive results in the literature.
We also produce some counterexamples to show that such a pushfoward need not
be a tilting sheaf.  
\end{abstract}

\maketitle

%
%
%
%

\section{Introduction}

A well known conjecture asserts that projective homogeneous spaces over $\CC$
have full strong exceptional collections, see \S 2 for definitions. There are
many positive results in this direction starting with \cite{Be},
\cite{Kapranov} and \cite{Kapranov2}. A survey of results in this
direction can be found in \cite{kp}. Over non-algebraically closed ground
fields, it is known that such full strong exceptional collections cannot
exist in general. Weakening the assumption on the ground field requires weakening the conclusion. In this direction, there are  several weakenings that could be considered. The simplest is to ask for just a full exceptional collection. It
is known, see \cite{na2} and \cite{na3}, that such collections cannot exist. Rather than consider
exceptional collections, we consider tilting sheaves, see \S2 for their definition
and relationship with exceptional collections.  

The purpose of this paper is to study pushforwards of tilting sheaves. 
We consider the following setup : given
a variety $Y$ with tilting sheaf $\sT$ defined over some prime subfield and another variety $X$ defined
over $k$ that is an $l/k$-form 
of $Y$. Here $l/k$ is a Galois extension of  fields. The pullback $\sT_l$ under the morphism $\pi_l^Y:Y_l\to Y$ induced by base change can be seen to be a tilting sheaf on $Y_l$ (see \mref{p:main}).
There is a projection
$$
p : Y_l\rightarrow X.
$$
We investigate when the pushforward $p_*(\sT_l)$ is a tilting sheaf on $X$.
In previous work a number of positive results were obtained, see \cite{y} and \cite{na} for
certain homogeneous varieties and towers of homogeneous varieties. 
In this work we give
a counterexample to show that these sheaves need not be tilting sheaves in general, see section \ref{s:outer}.

As mentioned in the previous paragraph, various positive results regarding tilting sheaves on
twisted forms of varieties
have been obtained in recent papers.
In \cite{blunk}, tilting sheaves are constructed on generalized Severi-Brauer varieties via
a different approach to that given in this paper. The thesis \cite{y}, constructs tilting bundles
on Severi-Brauer schemes and some arithmetic toric varieties using the procedure in this work.
More recently, these ideas have been extended to generalized Severi-Brauer schemes and 
positive characteristic in \cite{na}.

A more detailed overview of the paper follows.  In section 2, we discuss basic results
about tilting sheaves, generation in derived categories and exceptional collections.
A criterion for $p_*(\sT_l)$ to be a tilting sheaf on $Y$ is given, see (\ref{ss:descent}).
In section 3 we recall Kapranov's exceptional collection on flag varieties from (\cite{Kapranov2}).
The section ends by noting that Kapranov's exceptional collection produces a tilting sheaf
on any inner form of a partial flag variety, also known as Severi-Brauer flag
varieties. This generalises a result of \cite{blunk}, see also \cite{y} and \cite{na}.
In section 4, we show that if we consider outer forms of flag varieties  then the pushforward
does not produce a tilting sheaf. 
The final section shows how
a tilting sheaf can be constructed on certain kinds of arithmetic toric varieties.

%
%
%
%

\section*{Notations and conventions}

We will work over a ground field $k$ of characteristic 0. We need the characteristic 0 assumption
in order to make use of the theorem of Borel-Bott-Weil. We will have occasion to make use of
possibly non-commutative $k$-algebras. This notion means that $k$ is in the center of the algebra.
We will assume all rings have identities and all modules over them are unital.
We will assume throughout that $X$ is a smooth projective $k$-variety.

\section*{Acknowledgements}

We would like to thank the referee for numerous helpful comments and a careful reading of the
paper.

%
%
%
%

%
%
%
%

\section{Tilting sheaves and base change}


\subsection{Generation in derived categories}
Let $\bD$ be a triangulated category and $S$ a set of objects in $\bD$. We denote by $<S>$ the smallest full triangulated
category containing all the objects in $S$. We denote by $<S>^\kappa$ the smallest thick triangulated containing all the
objects in $S$. Note that thick subcategories are assumed to be full.

An object $C$ of $\bD$ is said to be \emph{compact} if $\Hom(C,-)$ commutes with direct sums. We denote by
$\bD^c$ the full subcategory of compact objects.

Given a set $S$ of objects of $\bD$ we define $S^\perp$ to be the full subcategory of $\bD$ consisting
of objects $A$ with $\Hom_\bD(E[i],A)=0$ for all $E\in S$ and $i\in \ZZ$. We say that $S$  
\emph{right spans} $\bD$ if $S^\perp=\{0\}$.

If $\bD^c$ right spans $\bD$ we say that $\bD$ is compactly generated.

\begin{theorem} \label{t:rn} (Ravenel and Neeman)
Let $\bD$ be a compactly generated triangulated category. Then a set of compact objects $S$ right spans
$\bD$ if and only if $<S>^\kappa =\bD^c$.
\end{theorem}

\begin{proof}
See \cite[Theorem 2.1.2]{BV}.
\end{proof}

Let $Y$ be a scheme.
We denote  the unbounded derived category of quasi-coherent sheaves on $Y$ by $\rD(\Qcoh(Y))$ and 
the bounded derived category of coherent sheaves by $\rD^b(Y)$.

\begin{proposition} \label{p:compactGenerator}
Let $Y$ be a quasi-compact, separated scheme. Then $\rD(\Qcoh(Y))$ is compactly generated.
\end{proposition}

\begin{proof}
See \cite[proposition 2.5]{N2}.
\end{proof}

A complex in $\rD(\Qcoh(Y))$ is said to be \emph{perfect} if it is locally quasi-isomorphic to a bounded complex
of free sheaves.

\begin{proposition} \label{p:compactPerfect} Recall, from our notations and conventions, that $X$ is a smooth projective variety.
Then $C\in\rD(\Qcoh(X))$ is compact if and only if $C$ is perfect.
\end{proposition}

\begin{proof}
See \cite[Lemma 3.5]{C}.
\end{proof}

\begin{proposition}\label{p:generates}
Let  $l/k$ be a finite field extension. We have a canonical morphism
$\pi_l: X_l\rightarrow X$. Suppose that $\sT$ is a locally free sheaf on $X$. Then $\sT^\perp=\{0\}$ if and only if $(\pi_l^*\sT)^{\perp}=\{0\}$.
\end{proposition}

\begin{proof}
First suppose that $\sT^{\perp}=\{0\}$. By \mref{t:rn}, \mref{p:compactGenerator} and \mref{p:compactPerfect}, we know that 
$<\sT>^\kappa =\bD^b(X)$ and that in order to show $(\pi_l^*\sT)^\perp=\{0\}$, it suffices to 
show that  $<\pi_l^*\sT>^\kappa = \bD^b(X_l)$. As the functor $\pi_{l}^*$ is exact we have that for each coherent sheaf $\sF$ on
$X_l$ that $(\pi_l)^*(\pi_{l})_*\sF\in (\pi_l)^*<\sT>^\kappa$. But then by exactness of $\pi_l^*$ we have that
$\sF\otimes_k l\in<\pi_l^*\sT>^\kappa$. The result follows as $\sF$ is a direct summand of $\sF\otimes_k l$.

Conversely assume that $<\pi_l^*\sT>^\kappa = \bD^b(X_l)$.  Consider the cartesian
square
\begin{center}
\begin{tikzpicture}
  \node (TL) {$X_l$};
  \node (BL) [below of=TL] {$\Spec(l)$};
  \node (TR) [right of=TL] {$X$};
  \node (BR) [right of=BL] {$\Spec(k)$};
  \draw[->] (TL) to node {$\pi_l$} (TR);
  \draw[->] (TL) to node {$q$} (BL);
  \draw[->] (BL) to node {$u$} (BR);
  \draw[->] (TR) to node {$p$} (BR);
\end{tikzpicture}
\end{center}
Suppose that $\sM\in \sT^\perp$. Then
\begin{align*}
0 &= u^*\bR\Hom(\sT,\sM) \\
 &=  u^*\bR p_*(\sT^\vee \otimes \sM) \\
 &=  \bR q_* \pi_l^*(\sT^\vee \otimes \sM) \\
 &= \bR \Hom(\pi_l^*\sT,\pi^*_l\sM).
\end{align*}
Hence $\pi^*_l\sM\in (\pi_l^*\sT)^{\perp}=\{0\}$. Finally $\sM=0$ as $\pi_l$ is faithfully flat.
\end{proof}


\subsection{Self extensions}

Recall that a coherent sheaf $\sF$ is said to have no \emph{higher self extensions} if $\Ext^i(\sF,\sF)=0$
for $i>0$.

\begin{lemma}\label{l:selfExt}
Let $l/k$ be a finite field extension. We have a canonical morphism
$\pi_l: X_l\rightarrow X$. If $\sT$ is a locally free coherent sheaf on $X$ then $\sT$ has no higher self extensions
if and only if $\pi_l^*\sT$ has no higher self extensions.
\end{lemma}

\begin{proof}
This follows via flat base change.
\end{proof}


\subsection{Tilting sheaves and base change}
We remind the reader that we have assumed that $X$ is a smooth projective $k$-scheme and $k$ is a field of characteristic 
$0$.

Recall the notion of a tilting object from \cite[1.1]{buchweitz}.

\begin{definition} An object $\sT^\bullet\in D(\Qcoh(X))$ is called a \emph{tilting object for} $X$ if
\begin{enumerate}[(i)]
	\item $\sT^\bullet$ is compact
	\item $\Hom(\sT^\bullet,\sT^\bullet[i])=0$ for $i\ne 0$ ($\sT^\bullet$ has
	no higher self extensions)
	\item ${\sT^\bullet}^\perp \cong \{0\}$.
\end{enumerate}
\end{definition}

A \emph{tilting sheaf} is a tilting object quasi-isomorphic to a coherent sheaf 
of finite rank
concentrated in cohomological degree 0.

\begin{theorem}
	Suppose $\sF$ is a tilting object for $X$. If $A=\End(\sF)$ then
\begin{enumerate}
	\item The functor ${\mathbb R}\Hom(\sF,-):D(\Qcoh(X)) \rightarrow D(\text{Mod}(A))$ 
	is an equivalence
	\item This equivalence restricts to an equivalence
	$$ D^b(X) \rightarrow {\rm perf}(A) $$
	\item If $X$ is smooth then $A$ has finite global dimension.
\end{enumerate}
\end{theorem}

\begin{proof}
\cite[7.6]{hilleVan}.
\end{proof}

\begin{proposition}\label{p:main}
Let be $l/k$ a finite field extension. Denote by
$\pi_l: X_l\rightarrow X$ the canonical morphism. Suppose that $\sT$ is a locally free sheaf on $X$.
Then $\sT$ is a tilting sheaf on $X$ if and only if  $\pi_l^*\sT$ is
a tilting sheaf on $X_l$.
\end{proposition}

\begin{proof}
Since $\sT$ and $\pi_l^*(\sT)$ are compact, the result follows from
\mref{l:selfExt} and  \mref{p:generates}.
\end{proof}

\subsection{Galois Descent}\label{ss:descent}

Consider a smooth projective variety $Y$ defined over the prime subfield $k^{\rm pr}$ of $k$. We assume that $\sT$ is a tilting sheaf on $Y$.

Consider a finite Galois extension $l/k$ with Galois group ${\rm Gal}(l/k)$. Let $X$ be 
an $l/k$ form of $Y$. This means that $X$ is a variety defined over $k$ 
and we have an $l$-isomorphism $\gamma:Y_l\cong X_l$. Both of the varieties 
$Y_l$ and $X_l$ have actions of ${\rm Gal}(l/k)$. 
Taking the ``difference'' of these two
actions produces a Galois cocycle 
\[
\phi_X: {\rm Gal}(l/k) \rightarrow {\rm Aut}_l(Y_l), \phi_X(g)=\gamma^{-1}\circ (g\gamma)
\]

\begin{remark}
Since $\sT$ is a tilting sheaf on $Y$, we have that $\sT_l=(\pi_l^Y)^*(\sT)$ is 
a tilting sheaf on $Y_l$ by \mref{p:main}.
Let
$$
p:Y_l\rightarrow X
$$
be the projection given by the composite $p=\pi_l^X\circ \gamma$ 
where $\pi_l^X:X_l\to X$ is the morphism induced by base change and $\gamma:Y_l\to X_l$ is the fixed $l$-isomorphism.
We are interested in determining when the pushforward $p_*(\sT_l)$ is a tilting
sheaf on $X$.
By \mref{p:main}, it suffices to determine whether
$\sF=p^*(p_*(\sT_l))$ is a tilting sheaf on $Y_l$.
Note that 
\[
\sF=p^*p_*(\sT_l)= \oplus_{g\in {\rm Gal}(l/k)} \phi_X(g)^*(\sT_l)
\]
descends to a sheaf on $X$. 

The sheaf $\sF=p^*p_*(\sT_l)$ is coherent and hence compact. By \mref{p:generates}, we have that $\sF=p^*p_*(\sT_l)^\perp\cong\{0\}$.

Hence to see if $p_*(\sT_l)$ is a tilting sheaf on $X$, it suffices to 
check that $\sF=p^*p_*(\sT_l)$ has no higher self extensions.
\end{remark}

In some cases the following result applies :
\begin{proposition}\label{p:easy}
In the above setting,
suppose that there is a  locally free tilting sheaf $\sT$ on $Y$. Suppose that
for each $g\in{\rm Gal}(l/k)$ we have
\[
\phi_X(g)^*(\sT_l) \cong \sT_l
\]
then there is a tilting sheaf $p_*(\sT_l)$ on $X$ obtained by pushing the tilting sheaf
on $Y_l$ forward along the projection
\[
p:Y_l\rightarrow X
\]
\end{proposition}

\begin{proof}
By the discussion in the remark above, it suffices to 
check that $p^*p_*(\sT_l)$ has no higher self extensions.
But by the calculation above and the hypothesis, we see that
$p^*p_*(\sT_l)\cong (\sT_l)^{[l:k]}$ has no higher self extensions since this 
is true for $\sT_l$.
\end{proof}

\subsection{Tilting sheaves and exceptional collections}

Many of the tilting sheaves in this work come from exceptional collections. We begin by recalling
the definition.

\begin{definition}
Let $\bD$ be a $k$-linear triangulated category. An object $E$  is said to be \emph{exceptional} if
 $$ \Hom(E, E)=k ~~~~ \text{and} ~~~~ \Hom(E, E[m])=0 ~~  \forall \,\, m \neq 0.$$
 An \emph{exceptional collection} in $\bD$ is an ordered collection $(E_0, E_1, \cdots,  E_n)$ of exceptional objects, satisfying
 $$  \Hom(E_j, E_i[m]) = 0 ~~ \text{for all}~~ m \text{ when } 0 \leq i < j \leq n.   $$
 If in addition
 $$ \Hom(E_j, E_i[m]) = 0 ~~ \text{for}~~ 0 \leq j \leq i \leq n,\, m \neq 0, $$
 we call $(E_0, E_1, \cdots,  E_n)$ a \emph{strong exceptional collection}. The collection is  \emph{full} (or \emph{complete}) if it generates $\bD$.
\end{definition}

\begin{lemma}\label{l:FSEcollection}
{Let $(\mathcal{F}_0, \mathcal{F}_1, \cdots,  \mathcal{F}_n)$ be a full strong exceptional collection of coherent sheaves on $X$, then $\mathcal{T} = \oplus_{i=0}^n \mathcal{F}_i^{\oplus l_i},\, l_i \geq 1$, is a tilting sheaf on $X$.
}
\end{lemma}

\begin{proof}
The second axiom follows from the fact that collection is a strong exceptional collection. Fullness of the collection amounts to axiom 3.
As $\sT$ is coherent, it is compact.
\end{proof}

%
%
%
%

\section{Partial Flag varieties}

For a fixed $k$-vector space we will denote by $F(d_1,\dots,d_s,V)$ the 
partial flag variety of flags 
\[
V_1\subseteq V_2\subseteq \ldots V_s \subseteq V
\]
with $\dim V_i =d_i$. The universal tautological flag will be denoted by
$$\sW_{d_1}^{\univ}\subseteq \sW_{d_2}^{\univ}\subseteq \dots 
\subseteq \sW_{d_s}^{\univ}.$$ 

\subsection{Kapranov's exceptional collection for partial flag varieties}

In \cite{Kapranov2} a complete exceptional collection for the partial
flag variety  $P=F(d_1,\dots,d_s,V)$ is constructed. In this subsection we will describe this
collection.

 Each such partial flag variety can be expressed as the composite of relative
Grassmann bundles.
Let $p_r:F(d_r,\dots,d_s,V)\to F(d_{r+1},\dots,d_s,V)$ be the natural fibration
with fibre $\gr(d_r,\sW^{\univ}_{d_{r+1}})$ for $r=1,\dots,s$
which we will identify with the relative Grassmann bundle
$$p_r: \gr(d_r,\sW^{\univ}_{d_{r+1}})\to F(d_{r+1},\dots,d_s,V)$$
For each $r=1,\dots,s$, let $\Gamma_r$ be the set of all  partitions corresponding to Young diagrams fitting 
into a box of size $d_r\times (d_{r+1}-d_r)$.
Then Kapranov's exceptional collection for the partial flag 
variety $F(d_1,\dots,d_s,V)$ is given
by 
$$\{\Sigma^{\alpha_1}(\sW^{\univ}_{d_1})\otimes \cdots \otimes  \Sigma^{\alpha_s}(\sW^{\univ}_{d_s}): \alpha_r\in \Gamma_r, 1\le r\le s\}$$
Note that this exceptional collection is built from the exceptional collection
on $\gr(d_s,V)$ using the sequence of relative Grassmann bundles used to determine
the partial flag variety on $V$.

\begin{theorem}\label{t:kapranov}
The sheaves $\Sigma^{\alpha_1} \sW_{d_1}^{\univ}\otimes \cdots \otimes \Sigma^{\alpha_s} \sW_{d_s}^{\univ}$ occurring in the above decomposition
form a complete, strong, exceptional collection for the partial flag variety
$F(d_1,\dots,d_s,V)$.
\end{theorem}

\begin{proof}
See \cite[Prop. 3.9]{Kapranov2}.
\end{proof}

\subsection{Twisted Automorphisms of General Flag Varieties}

Let $V$ be an $n$ dimensional $k$-vector space. Given $1\le d_1< \dots<d_s\le n$,
we denote by $P=F(d_1,\dots,d_s,V)$ the variety of partial flags of type
$(d_1,\dots,d_s)$ in the $n$ dimensional vector space $V$.
When we want to make the base field clear we will write $F(d_1,\dots,d_s,V)_k$ or $P_k$.
Recall
that the partial flag variety is a moduli space.
As such, there are universal exact sequences
\[
0\rightarrow \sW_{d_1}^{\univ}\hookrightarrow \cdots \hookrightarrow \sW_{d_s}^{\univ} \hookrightarrow \sO_{F}\otimes V\twoheadrightarrow \sQ_{d_1}^{\univ}\twoheadrightarrow \cdots \twoheadrightarrow \sQ_{d_s}^{\univ} \rightarrow 0.
\]
We begin by recalling the structure of $\Aut_k(F(d_1,\dots,d_s,V))$. Any $\phi\in\gl(V)$ induces
new universal exact sequences by
\[
0\rightarrow \sW_{d_1}^{\univ}\hookrightarrow \cdots \hookrightarrow \sW_{d_s}^{\univ} \rightarrow \sO_{F}\otimes V\stackrel{1\otimes\phi}\twoheadrightarrow \sQ_{d_1}^{\univ}\twoheadrightarrow \cdots \twoheadrightarrow\sQ_{d_s}^{\univ} \rightarrow 0
\]
and hence determines an automorphism of $F(d_1,\dots,d_s,V)$. This gives an inclusion
$\pgl(V)\hookrightarrow \Aut_k(F(d_1,\dots,d_s,V)$. In most cases this completely describes the
automorphism group. When $d_i+d_{s-i+1}=n$ for all $1\le i\le s$ there is one more automorphism.
Choose an isomorphism $V\cong V^\vee$.  This induces an automorphism 
$\sigma$ of $F(d_1,\dots,d_s,V)$ sending the above universal exact sequences
to 
\[
0\rightarrow (\sQ_{d_s}^{\univ})^{\vee}\hookrightarrow \cdots \hookrightarrow (\sQ_{d_1}^{\univ})^{\vee} \hookrightarrow \sO_{F}\otimes V\twoheadrightarrow (\sW_{d_1}^{\univ})^{\vee}\twoheadrightarrow \cdots (\sW_{d_s}^{\univ})^{\vee} \rightarrow 0.
\]
So in particular, 
$\sigma^*(\sW_{d_i}^{\univ})\cong (\sQ_{d_{s-i+1}}^{\univ})^{\vee}$, for all $i=1,\dots,s$.

\begin{theorem}\label{t:autFlag}
\begin{enumerate}[(i)]
\item Suppose that there exists $i$ with $d_i+d_{s-i+1}\ne n$.
 Then $\Aut_k(F(d_1,\dots,d_s,V)) = \pgl(V)$.
\item Suppose that for all $1\le i\le s$, we have $d_i+d_{s-i+1}=n$. Then 
$$\Aut_k(F(d_1,\dots,d_s,V))=\langle \pgl(V),\sigma\rangle.$$
\end{enumerate}
\end{theorem}

\begin{proof}
This theorem is due to Chow in characteristic 0, see \cite{chow:49}. In arbitrary
characteristic a proof can be found in \cite{tango:76}.
\end{proof}

The scheme $F(d_1,\dots,d_s,V)$ can be defined over $\ZZ$, along with its universal exact sequences. Hence for
each field $k$ and each automorphism $\alpha$ of $F(d_1,\dots,d_s,V)$ over $\ZZ$ lifts canonically to an automorphism, also
denoted $\alpha$, of $F(d_1,\dots,d_s,V)$ over $k$.

\begin{proposition}\label{p:twistedPullback}
In the above setting we have $\alpha^*(\sW^{\univ})\cong\sW^{\univ}$.
\end{proposition}

\begin{proof}
This is because $\sW^{\univ}$ descends to a sheaf over $F(d_1,\dots,d_s)_{\ZZ}$.
\end{proof}

\begin{corollary}\label{c:twistedPullback}
Let $\phi:F(d_1,\dots,d_s,V)\rightarrow F(d_1,\dots,d_s,V)$ be a twisted automorphism.
\begin{enumerate}[(i)]
\item If $\dim V\neq d_i+d_{s-i+1}$ for some $i=1,\dots,s$,
 then $\phi^*(\sW_{d_i}^{\univ})\cong \sW_{d_i}^{\univ}$ for all $i=1,\dots,s$.
\item If $\dim V= d_i+d_{s-i+1}$ for all $i=1,\dots,s$, then either 
$\phi^*(\sW_{d_i}^{\univ})\cong \sW_{d_i}^{\univ}$ or
$\phi^*(\sW_{d_i}^{\univ})\cong (\sQ_{d_{s-i+1}}^{\univ})^{\vee}$
\end{enumerate}
\end{corollary}

\begin{proof}
After writing $\phi = \psi \circ \alpha$ where $\psi$ is an automorphism of $F(d_1,\dots,d_s,V)$
and $\alpha$ is an automorphism of $k$ the result follows from the above discussion.
\end{proof}

\subsection{Severi-Brauer Flag varieties}
For a general introduction to generalised Severi-Brauer varieties we refer the reader to
\cite{panin}, especially sections 4 and 5.

Consider a Severi-Brauer flag variety $X=\SB(d_1,\dots,d_s,A)\rightarrow \spec(k)$ where $A$ is a $k$-central simple algebra.
Such an $X$ is an inner form of a partial flag variety. That is, there is a cartesian square of the
form

\begin{center}
\begin{tikzpicture}
  \node (TL) {$F(d_1,\dots,d_s,V)$};
  \node (BL) [below of=TL] {$\Spec(l)$};
  \node (TR) [right of=TL] {$X$};
  \node (BR) [right of=BL] {$\Spec(k),$};
  \draw[->] (TL) to node { } (TR);
  \draw[->] (TL) to node { } (BL);
  \draw[->] (BL) to node { } (BR);
  \draw[->] (TR) to node { } (BR);
\end{tikzpicture}
\end{center}
where $l/k$ is a Galois extension and the $1$-cocycle
$$
{\rm Gal}(l/k)\rightarrow {\rm Aut}(F(d_1,d_2,\ldots, d_s,V))
$$
factors through ${\rm PGL(V)}$.

\begin{theorem}
$\SB(d_1,\dots,d_s,A)$ has a locally free tilting sheaf.
\end{theorem}

\begin{proof} We can just apply (\ref{p:easy}) as it is clear that an inner automorphism preserves the sheaves in the exceptional collection.
\end{proof}

\section{Outer forms of Partial Flag Varieties.}\label{s:outer}

In this section we consider twisted forms of partial flag varieties
$$P=F(d_1,\dots,d_s,V), d_i+d_{s-i+1}=\dim(V)=n,i=1,\dots s$$
such that the associated Galois cocycle
$$
{\rm Gal}(l/k)\rightarrow {\rm Aut}_l(F)
$$
does not factor through ${\rm PGL}(V)$. The associated form $X$ of $F$ is called an outer form. $X$ can be realised as a Severi-Brauer flag variety $\SB(d_1,\dots,d_s,A)$
for $A$ a  $k$ central simple algebra equipped with a unitary involution.

In this case our method does not produce a tilting sheaf. This does not mean a tilting
sheaf does not exist although, to the best of our knowledge, no such sheaf exists in the literature
at this time.

In this setting, the partial flag variety has an extra automorphism
$\sigma$ that sends the tautological flag 
$$\sW_{d_1}^{\univ}\subseteq \dots \subseteq \sW_{d_s}^{\univ}$$
to
$$(\sQ_{d_s}^{\univ})^{\vee}\subseteq \dots \subseteq (\sQ_{d_1}^{\univ})^{\vee}$$
Let $\sE=\Sigma^{\alpha_1}(\sW_1^{\univ})\otimes\cdots\otimes \Sigma^{\alpha_s}(\sW_s^{\univ})$
be a bundle in Kapranov's exceptional collection.
Then the image under the extra automorphism is 
$\sigma^*(\sE)=\Sigma^{\alpha_1}((\sQ_{d_s}^{\univ})^{\vee})\otimes\cdots \otimes \Sigma^{\alpha_s}((\sQ_1^{\univ})^{\vee})$.

We show that when $d_i+d_{s-i+1}=n$ for all $i$,
the image of Kapranov's
exceptional collection under the automorphism group of $F(d_1,\dots,d_s,V)$
cannot be an exceptional collection, since, in particular, higher Ext groups do not vanish.
In other words, we  will produce
 bundles $\sF$ and $\sG$ in Kapranov's exceptional
collection such that
$$\Ext_P^i(\sigma^*(\sF),\sG)\ne 0$$ for some $i>0$.
 
We first discuss the methods behind our calculations.
Let $\sE=\sigma^*(\sF)^{\vee}\otimes \sG$.
Then as the exceptional collection consists of vector bundles, we have 
$$\Ext_P^*(\sigma^*(\sF),\sG)=H^*(P,\sE)$$
Also, we may factor the structure morphism  $p$ of $F(d_1,\dots,d_s,V)$
as a sequence of relative Grassmannian bundles
$$p_i:F(d_i,\dots,d_s,V)\to F(d_{i+1},\dots,d_s,V), i=1,\dots,s-1$$
with $p_s$ being the structure morphism for $F(d_s,V)=\gr(d_s,V)$.
Here we identify $p_i$ with the relative Grassmann bundle
$$p_i:\gr(d_i,\sW^{\univ}_{d_{i+1}})\to F(d_{i+1},\dots,d_s,V)$$
Then, since $p=p_s\circ \cdots \circ p_1$, we see that, in the 
derived category, we have
$$Rp_*(\sE)=R(p_s)_*\circ \cdots \circ R(p_1)_*(\sE)$$

Let $\sE_i=R(p_i)_*\circ R(p_{i-1})_*\circ \cdots R(p_{1})_*(\sE)$ for 
$i=1,\dots,s$ and $\sE_0=\sE$.
At each stage i, we wish to reexpress 
$\sE_i$ in terms of bundles of the form 
$$R(p_i)_*(\Sigma^{\alpha}(\sW_{d_{i+1}}/\sW_{d_{i}})\otimes \Sigma^{\beta}(\sW_{d_{i}}))\otimes \sE_{i+1}'$$
where $\sE_{i+1}'$ is a bundle defined over $F(d_{i+1},\dots,d_s,V)$.
To do this, we 
use exact sequences of bundles derived from the natural sequences
$$0\to \sW_{d_{i+1}}^{\univ}/\sW_{d_{i}}^{\univ}\to \sQ_{d_i}^{\univ}\to \sQ_{d_{i+1}}^{\univ}\to 0$$
We will make use  of the tools discussed in the next subsection, particularly
Proposition~\ref{grasstoflag}, relative Borel-Bott-Weil and the projection formula to 
determine $\sE_i$ from $\sE_{i-1}$
as a bundle of $F(d_{i+1},\dots,d_s,V)$. 

\subsection{Cohomological Tools}

Fix a Borel subgroup $B\subseteq {\rm GL}_n$.
The character group of $B$, $X(B)$ is the character lattice $X(T)$ of the maximal
torus $T$ and so is in bijection with $\ZZ^n$. Indeed. 
$$X(B)=X(T)=\langle \chi_i:i=1,\dots,n\rangle\cong \ZZ^n$$
where $\chi_i$ is the $i$th projection.
The dominant Weyl chamber $C^+$ consists of sequences $\chi=(a_1,a_2,\ldots, a_n)$
with $a_1\ge a_2\ge \ldots \ge a_n$. The irreducible representations
of $GL(V)$ are given by $\Sigma^{\chi}(V)$ for each $\chi\in C^+$
 where  $\Sigma^{\chi}$ is the corresponding Schur functor.
Note that $(\Sigma^{\chi}(V))^{\vee}=\Sigma^{-\chi}(V)$ for $\chi\in C^+$
where $-\chi=(-a_n,\dots,-a_1)\in C^+$ if $\chi=(a_1,\dots,a_n)\in C^+$.
There is an action of the Weyl group $S_n$ given by permutation of
letters.
We denote half the sum of the positive roots by $\rho=(n,n-1,\ldots,1)$.
There is a modified action of the Weyl group $S_n$ on the weights $\ZZ^n$ given
by 
\[
\sigma . \lambda = \sigma(\lambda + \rho) - \rho.
\]

Let $\sV$ be a vector bundle of rank $n$ over a scheme $X$ and 
$\pi:\flag(\sV)\to X$ be the relative full flag bundle over $X$.
Note that there is a $\gl_n$-torsor $T(\sV)=\Isom(\sO_X^n,\sV)$
over $X$.  The fibre over a point $x\in X$ is the set of frames at $x\in X$,
$\Isom(k^n,\sV_x)$ on which $\gl_n$ acts freely by precomposition.
Then $T(\sV)/B\cong \flag(\sV)$.
Each character of $B$, $\chi\in \ZZ^n$
produces a line bundle 
\[
\sO_F(\chi) \cong T(\sV)\times_{B,\chi} {\mathbb G}_m
\]
where $F=\flag(\sV)$.
If $\chi=(\beta_1,\beta_2,\ldots, \beta_n)$ then 
\[
\sO_F(\chi)\cong \sW_1^{-\beta}\otimes (\sW_2/\sW_1)^{-\beta_2} \otimes \ldots \otimes
(\sV/\sW_{n-1})^{-\beta_n}.
\]

The  Borel-Bott-Weil Theorem 
determines $R\pi_*(\sO_F(\chi))$ for $\chi\in C_+$.

\begin{theorem} ( Borel-Bott-Weil)\label{bbw}
Let $\sV$ be a vector bundle over a scheme $X$ and $\pi:\flag(\sV)\to X$
be the relative full flag bundle over $X$.
Let 
\[
0=\sW_0\subseteq \sW_1\subseteq \sW_2\subseteq \ldots \subseteq \sW_n = 
\sV.
\]
be a universal flag on $F=\flag(\sV)$.
For $\beta=(\beta_1,\dots,\beta_n)\in \ZZ^n$, we define a corresponding line
bundle on $\flag(\sV)$
$$\sO_F(\beta)=\sW_1^{\otimes -\beta_1}\otimes (\sW_2/\sW_1)^{\otimes -\beta_2}\cdots \otimes
(\sV/\sW_{n-1})^{-\beta_n}$$
Then for $\chi\in \ZZ^n$:
\begin{enumerate}
\item 
If there exists a non-identity $w\in S_n$ such that
$w\cdot \chi=\chi$ (or equivalently if there is a repeat in $\chi+\rho$)
then $R^i\pi_*(\sO_F(\chi))=0$ for all $i$.
\item Otherwise, there exists a unique $w\in S_n$ such that $\alpha=w\cdot \chi\in C^+$.
In this case, if $i\ne l(w)$, we have $R^i\pi_*(\sO_F(\chi))=0$ and 
$R^{l(w)}\pi_*(\sO_F(\chi))=\Sigma^{\alpha}(\sV)^{\vee}=\Sigma^{-\alpha}(\sV)$.
\end{enumerate}
\end{theorem}

\begin{proof}
This result is well known and there are many reference. Two such references are
\cite[page 392]{harris} and \cite[page 217]{jantzen}
\end{proof}

We will  be interested in relative Grassmann bundles
over a scheme $X$.  Let $\sV$ be a bundle over $X$ and let
$p:\gr(k,\sV)\to X$ be the relative Grassmann bundle and $\pi:\flag(\sV)\to X$.
We wish to express the higher derived functors of $p$ for certain bundles
over $\gr(k,\sV)$ in terms of the higher derived functors of $\pi$ for 
certain line bundles over $\flag(\sV)$.
This proposition follows  from the discussion in \cite{Kapranov}.
\begin{proposition}\label{grasstoflag}
Suppose we have decreasing sequences 
\[
\alpha = (\alpha_1\ge \alpha_2\ge \ldots \ge \alpha_k)\quad\text{and}\quad
\beta = (\beta_1\ge \beta_2 \ge \ldots \ge \beta_{n-k}).
\]
Let $\sV$ be a bundle on a scheme $X$, and let
$p:\gr(k,\sV)\to X$ be the relative Grassmann bundle on $\sV$ and let
 $\pi:\flag(\sV)\to X$ be the full flag variety.
Let $\sW$ be the tautological subbundle on $\gr(k,\sV)$.  
Then there is a cartesian  diagram
\begin{center}
\begin{tikzpicture}
\node (TL) {$\flag(\sV)$};
  \node (BL) [below of=TL] {$\flag(\sW)$};
  \node (TR) [node distance=3.4cm,right of=TL] {$\flag(\sV/\sW)$};
  \node (BR) [node distance=3.4cm,right of=BL] {$\gr(k,\sV)$};
  \draw[->] (TL) to node {$ $} (TR);
  \draw[->] (TL) to node {$ $} (BL);
  \draw[->] (BL) to node {$q_1$} (BR);
  \draw[->] (TR) to node {$q_2$} (BR);
\end{tikzpicture}
\end{center}
Further $R\pi_*(\sO_{F}(-\alpha_k,\dots,-\alpha_1,-\beta_{n-k},\dots,-\beta_1)
=Rp_*(\Sigma^{\alpha}(\sW)\otimes \Sigma^{\beta}(\sV/\sW))$.
\end{proposition}

\begin{proof}
The statement on the cartesian diagram follows immediately from the description of the flag varieties
as moduli spaces.

Let $\sL_1=\sO_{F_1}(-\alpha_1)$ and $\sL_2=\sO_{F_2}(-\alpha_2)$ be 
line bundles on $F_1=\flag(\sW)$ and $F_2=\flag(\sV/\sW)$ respectively.
By  Borel-Bott-Weil,
we see that $(q_1)_*(\sL_1)=(q_1)_*(\sO_{F_1}(-\alpha))=(\Sigma^{-\alpha}(\sW))^{\vee}
=\Sigma^{\alpha}(\sW)$  and $R^i(q_1)_*(\sL_1)=0$ for $i>0$ since $-\alpha$ is dominant if $\alpha$ is dominant.
Similarly, we see that $(q_2)_*(\sL_2)=(q_2)_*(\sO_{F_2}(-\beta))=\Sigma^{\beta}(\sV/\sW)$.
Since 
\[
0=\sW_0\subseteq \sW_1\subseteq \sW_2\subseteq \ldots \subseteq \sW_n = 
\sV.
\]
is a universal flag for the relative full flag bundle $F=\flag(\sV)$
with projection $q:F=\flag(\sV)\to G=\gr(k,\sV)$, we see that $q=q_1\times_G q_2$ and $\flag(\sV)=
\flag(\sW)\times_G\flag(\sV/\sW)$.
Then $\sL=\sL_1\otimes \sL_2=\sO_{\flag(\sW)}(-\alpha)\otimes_G \sO_{\flag(\sV/\sW)}(-\beta)=\sO_{\flag(\sV)}(-\alpha_k,\dots,-\alpha_1,-\beta_{n-k},\dots,-\beta_1)$.
By the  K\"unneth formula, $q_*(\sL_1\otimes \sL_2)=(q_1)_*(\sL_1)\otimes (q_2)_*(\sL_2)=\Sigma^{\alpha}(\sW)\otimes \Sigma^{\beta}(\sV/\sW)$
and $R^iq_*(\sL_1\otimes \sL_2)=0$ for $i>0$.

Now $\pi=p\circ q$ where $p:\gr(k,\sV)\to X$ and $\pi:\flag(\sV)\to X$.
By the Leray spectral sequence, we see that
$R\pi_*(\sL)=Rp_*\circ Rq_*(\sL)=Rp_*(\Sigma^{\alpha}(\sW)\otimes \Sigma^{\beta}(\sV/\sW))$ as required.

\end{proof}

\begin{corollary}
Let $\sV$ be a bundle on $X$ and  $p:\gr(k,\sV)\to X$,
the relative Grassmann bundle.  Set $G=\gr(k,\sV)$. Then
$Rp_*(\sO_G)=\sO_X$.
\end{corollary}

\begin{proof}
$\sO_G=\Sigma^0(\sW)\otimes \Sigma^0(\sV/\sW)$ where $\sW$ is the tautological
bundle on $G$ and $\sV/\sW$ is the tautological quotient bundle.
Then by the proposition and relative Borel-Bott-Weil, 
$$Rp_*(\sO_G)=R\pi_*(\sO_F(0))=\Sigma^0(\sV)=\sO_X$$
where $F=\flag(\sV)$.
\end{proof}

Recall also the projection formula:
\begin{proposition}
Let $p:Y\to X$ be a morphism of schemes. Let $\sE$ be a bundle on $Y$ 
and let $\sF$ be a bundle on $X$.
Then $Rp_*(\sE\otimes p^*(\sF))=Rp_*(\sE)\otimes \sF$.
\end{proposition}

Lastly, we recall a filtration on exterior algebra bundles determined by
a short exact sequence which will later prove helpful:

\begin{proposition}
Let $0\to \sF'\to \sF\to \sF''\to 0$ be an exact sequence of locally free 
sheaves on a scheme $X$.
Then for any $r$, there is a finite filtration of $\bigwedge^r(\sF)$,
$$\bigwedge^r(\sF)=F^0\supseteq F^1\supseteq F^2\supseteq \dots \supseteq
F^r\supseteq F^{r+1}=0$$
with quotients
$$F^p/F^{p+1}\cong \bigwedge^p(\sF')\otimes \bigwedge^{r-p}(\sF'')$$
for each p.
\end{proposition}
\begin{proof}
Exercise II 5.16(c) in Hartshorne.
\end{proof}

\begin{corollary}\label{extalgses}
Let $0\to \sF'\to \sF\to \sF''\to 0$ be an exact sequence of locally free 
sheaves on a scheme $X$ where $\sF''$ has rank 1. 
Then we obtain an exact sequence 
$$0\to \bigwedge^r(\sF')\to \bigwedge^r(\sF)\to \bigwedge^{r-1}(\sF')\otimes \sF''\to 0$$
\end{corollary}

\begin{proof}
From the proposition, there is a filtration on $\bigwedge^r(\sF)$ given by
$$\bigwedge^r(\sF)=F^0\supseteq F^1\supseteq F^2\supseteq \dots \supseteq
F^r\supseteq F^{r+1}=0$$
with quotients
$$F^p/F^{p+1}\cong \bigwedge^p(\sF')\otimes \bigwedge^{r-p}(\sF'')$$
for each $p$.
But in our case, $F^p/F^{p+1}$ vanishes for all $p=0,\dots,r-2$ since $\bigwedge^{r-p}(\sF'')=0$. So we have $\bigwedge^r(\sF)=F^0=\dots=F^{r-1}$ and $F^{r+1}=0$. 
This means that the natural exact sequence 
$$0\to F^r\to F^{r-1}\to F^{r-1}/F^r\to 0$$
gives the exact sequence
$$0\to \bigwedge^r(\sF')\to \bigwedge^r(\sF)\to \bigwedge^{r-1}(\sF')\otimes \sF''\to 0$$
as required.
\end{proof}

\subsection{Non-vanishing Ext groups}\label{ss:non}
In this subsection we consider a partial flag variety $P=F(d_1,d_2,\ldots, d_s,V)$ such that $d_i+d_{s-i+1}=n=\dim(V), i=1,\dots,s$. Such a flag variety
has an extra automorphism $\sigma$. 

We will prove the following proposition.

\begin{proposition} Suppose $d_i+d_{s-i+1}=n=\dim(V), i=1,\dots,s$. 
Let $l/k$ be a finite Galois extension.
Let $X$ be an outer $l/k$ form of $P=F(d_1,\dots,d_s,V)$ where $l/k$ is a finite Galois extension.
Let $\sT$ be the tilting sheaf on $P$ corresponding to Kapranov's exceptional 
collection for $P$. Let $\sT_l=\pi_l^*(\sT)$ be the pullback of $\sT$ under $\pi_l^P:P_l\to P$.  Then the pushforward $p_*(\sT_l)$ under the projection
$p:Y_l\to X$ is not a tilting sheaf on $X$.
\end{proposition}

\begin{proof}
We will divide our argument into 3 cases on the parameters $d_i,i=1,\dots,s$
in the partial flag variety $P=F(d_1,\dots,d_s,V)$.
For each case, we will show that 
there exists sheaves
$\sE$ and $\sF$ in Kapranov's exceptional collection so that
$$\Ext_P^i(\sigma^*(\sF),\sG)\ne 0$$ for some $i>0$. 
Then, by the discussion in Section~\ref{ss:descent},  for the tilting sheaf $\sT$ on $P$ produced from Kapranov's exceptional collection
via Lemma~\ref{l:FSEcollection}, the pushforward $p_*(\sT_l)$ has  non-trivial higher self-extensions and so is not a tilting sheaf on the outer form $X$.

We will simplify notation a little.
The universal subbundle 
$\sW_{d_i}^{\univ}$ and universal quotient bundles $\sQ_{d_i}^{\univ}$ by $\sW_{d_i}$ and $\sQ_{d_i}$.  We will also implicitly identify these bundles $\sW_{d_j},\sQ_{d_j}$
over $F(d_j,\dots,d_s,V)$ with their pullbacks to $F(d_i,\dots,d_s,V)$ where
$i<j$. Given a partition, we will often drop trailing
zeroes. For example the partition $(2)$ is really the partition $(2,0,\ldots,0)$. Further 
repeated entries in a partition will be indicated by superscripts, for example, $(1^d)$ is the
partition $(1,1,\ldots, 1)$ repeated $d$-times.

The construction is divided into three cases. 

\noindent\textbf{Case 1: $d_1\ge 2$}

Note that as we have assumed that $d_i+d_{s-i+1}=n$ for all $i=1,\dots,s$,
this implies that $d_{s}=n-d_1$.
Take $\sF=\Sigma^{(1^{d_1-1})}(\sW_{d_1})$ and $\sG=\Sigma^{(2)}(\sW_{d_{s}})$.
Note that 
$$\sF=\Sigma^{\alpha_1}(\sW_{d_1})\otimes \cdots \otimes \Sigma^{\alpha_s}(\sW_{d_s})$$
$$\sG=\Sigma^{\beta_1}(\sW_{d_1})\otimes \cdots \otimes \Sigma^{\beta_s}(\sW_{d_s})$$
where $\alpha_1=(1^{d_1-1})$ and $\alpha_i=0$ for all $i\ne 1$
and $\beta_s=(2)$, $\beta_i=0$ for all $i\ne s$.
Since $d_1-1\le d_1$ and $n-d_s=d_1\ge 2$, these vector bundles are part of the exceptional collection
constructed in (\ref{t:kapranov}).
Then 
$$\Ext^*_P(\sigma^*(\sF),\sG)=H^*(P,\Sigma^{(1^{d_1-1})}(\sQ_{d_s})\otimes \Sigma^{(2)}(\sW_{d_s}))$$

We factor the structure morphism of $P=F(d_1,\dots,d_s,V)$
into the projection $q:F(d_1,\dots,d_s,V)\to \gr(d_s,V)$ and the structure morphism $p$ for 
$\gr(d_s,V)$.
Then 
\begin{eqnarray*}H^*(P,\Sigma^{(1^{d_1-1})}(\sQ_{d_s})\otimes \Sigma^{(2)}(\sW_{d_s}))&=&
Rp_*(Rq_*(\Sigma^{(1^{d_1-1})}(\sQ_{d_s})\otimes \Sigma^{(2)}(\sW_{d_s}))\\
=Rp_*(\Sigma^{(1^{d_1-1})}(\sQ_{d_s})\otimes \Sigma^{(2)}(\sW_{d_s}))
\end{eqnarray*}
where the last line follows from the projection formula as our bundle
is defined over $\gr(d_s,V)$. Then for the structure morphism $\pi$ of
$\flag(V)$, we may apply Proposition~\ref{grasstoflag} to obtain
$$R\pi_*(\sO_{\flag(V)}(\chi))=H^*(\flag(V),\sO_{\flag(V)}(\chi))$$
where $\chi=(0,\dots,0,-2,0,-1,\dots,-1)$ has the $-2$ in the $d_s$th spot,
0 in the $d_s+1$ spot and the remaining entries -1, since $\pi$ is the structure morphism of $\flag(V)$.
For the simple transposition $w=(d_s,d_s+1)\in S_n$ of length 1, we see that 
$\alpha=w\cdot \chi=(0,\dots,0,-1,\dots,-1)$ is dominant
where the last $d_1+1$ entries are $-1$.  By Borel-Bott-Weil, we obtain
$$H^1(\flag(V),\sO_{\flag(V)}(\chi))=\Sigma^{(1^{d_1+1})}(V)$$
So following the chain of isomorphisms, we find that 
$$\Ext^1_P(\sigma^*(\sF),\sG)=\Sigma^{(1^{d_1+1})}(V)\ne 0$$
so that we have found a bundle of Kapranov's exceptional collection
and the image of a bundle of Kapranov's exceptional collection which
have non-trivial Ext group.

\noindent\textbf{Case 2: $d_1=1, d_2\ge 3$}

Note that $d_{s-1}\ge d_2\ge 3$ and by the symmetry assumption, we
have $d_s=n-1$ and $d_{s-1}=n-d_2$.
Take $\sF=\Sigma^{(1^{d_2-1})}(\sW_{d_2})$ and $\sG=\Sigma^{(2)}(\sW_{d_{s-1}})$.
Note that 
$$\sF=\Sigma^{\alpha_1}(\sW_{d_1})\otimes \cdots \otimes \Sigma^{\alpha_s}(\sW_{d_s})$$
$$\sG=\Sigma^{\beta_1}(\sW_{d_1})\otimes \cdots \otimes \Sigma^{\beta_s}(\sW_{d_s})$$
where $\alpha_2=(1^{d_2-1})$ and $\alpha_i=(0)$ for all $i\ne 2$
and $\beta_{s-1}=(2)$, $\beta_i=(0)$ for all $i\ne s-1$.
Since $d_{2}\ge 3$ and $d_{s}-d_{s-1}=(n-1)-(n-d_2)=d_2-1\ge 2$, these
vector bundles are part of the exceptional collection constructed in (\ref{t:kapranov}).
Then 
$$\Ext^*_P(\sigma^*(\sF),\sG)=H^*(P,\sE)$$
where $\sE=\Sigma^{(1^{d_2-1})}(\sQ_{d_{s-1}})\otimes \Sigma^{(2)}(\sW_{d_{s-1}})$.
We factor the structure morphism of $P=F(d_1,\dots,d_s,V)$
into $q:F(d_1,\dots,d_s,V)\to F(d_{s-1},d_s,V)$ and the structure morphism $t$ for 
$F(d_{s-1},d_s,V)$.
Then 
$$H^*(P,\sE)=
Rt_*(Rq_*(\sE))
=Rt_*(\sE)=H^*(F(d_{s-1},d_s,V),\sE)$$
where we use the projection formula and the fact that our bundle $\sE$
is defined over $F(d_{s-1},d_s,V)$.
We now factor the structure morphism $t$ for $F(d_{s-1},d_s,V)$
into the relative Grassmann bundle $p_{s-1}:F(d_{s-1},d_s,V)\to F(d_{s},V)=\gr(d_s,V)$ and the structure morphism $p_s$ for $\gr(d_s,V)$.
So 
$$H^*(F(d_{s-1},d_s,V),\sE)=R(p_s)_*(R(p_{s-1})_*(\sE))$$
We now analyse $R(p_{s-1})_*(\sE)$:
Since $\sE=\Sigma^{(1^{r})}(\sQ_{d_{s-1}})\otimes \Sigma^{(2)}(\sW_{d_{s-1}})$,
where $r=d_2-1=d_s-d_{s-1}$,
we need to reexpress the bundle $\Sigma^{(1^r)}(\sQ_{d_{s-1}})$ in terms of 
Schur functors of the bundles $\sW_{d_s}/\sW_{d_{s-1}}$
and $\sQ_{d_s}$.
Note that there is a natural exact sequence of bundles:
$$0\to \sW_{d_s}/\sW_{d_{s-1}}\to \sQ_{d_{s-1}}\to \sQ_{d_s}\to 0$$
Let $\sB=\sQ_{d_{s-1}}$, $\sB'=\sW_{d_s}/\sW_{d_{s-1}}$ and $\sB''=\sQ_{d_s}$.
Then since the Schur functor $\Sigma^{(1^r)}$ is in fact $\bigwedge^r$
and by assumption $\sB''$ has rank 1, we may use Corollary~\ref{extalgses} to obtain an 
exact sequence
$$0\to \bigwedge^r(\sB')\to \bigwedge^r(\sB)\to \bigwedge^{r-1}(\sB')\otimes \sB''\to 0$$
So we have 
$$0\to \bigwedge^r(\sW_{d_s}/\sW_{d_{s-1}})\to \bigwedge^r(\sQ_{d_{s-1}})\to
\bigwedge^{r-1}(\sW_{d_s}/\sW_{d_{s-1}})\otimes \sQ_{d_s}\to 0$$
Tensoring this with $\Sigma^{(2)}(\sW_{d_{s-1}})$, we get
$$0\to \sE'\to \sE \to \sE''\to 0$$
where 
$$\sE'=\Sigma^{(2)}(\sW_{d_{s-1}})\otimes \bigwedge^r(\sW_{d_s}/\sW_{d_{s-1}})
\mbox{ and }\sE''=\Sigma^{(2)}(\sW_{d_{s-1}})\otimes\bigwedge^{r-1}(\sW_{d_s}/\sW_{d_{s-1}})\otimes \sQ_{d_s}.$$
We wish to compute $R(p_{s-1})_*(\sE)$.
We note that 
$$0\to R(p_{s-1})_*(\sE')\to R(p_{s-1})_*(\sE)\to R(p_{s-1})_*(\sE'')\to 0$$
in the derived category. Using \mref{grasstoflag}, we have that
$$R(p_{s-1})_*(\sE')=R(\pi_{s-1})_*(\sO_{\flag(\sW_{d_s})}(\chi))$$
where $\pi_{s-1}:\flag(\sW_{d_s})\to F(d_s,V)$ is the relative full flag bundle
and $\chi=(0,\dots,-2,-1,\dots,-1)\in \ZZ^{n-1}$ has a $-2$ in the $d_{s-1}$ position followed by a string of $r$ (-1)s.
Adding $\rho=(n-1,\dots,1)$ to $\chi$, we find that we have $n-d_{s-1}-2$
in both $d_{s-1}$ and $d_{s-1}+1$ positions and so 
$R(p_{s-1})_*(\sE')=0$ by the relative version of Borel-Weil-Bott.

We now calculate 
\begin{eqnarray*}R(p_{s-1})_*(\sE)&=&R(p_{s-1})_*(\Sigma^{(2)}(\sW_{d_{s-1}})\otimes\bigwedge^{r-1}(\sW_{d_s}/\sW_{d_{s-1}})\otimes \sQ_{d_s})\\
&=&R(p_{s-1})_*(\Sigma^{(2)}(\sW_{d_{s-1}})\otimes\bigwedge^{r-1}(\sW_{d_s}/\sW_{d_{s-1}}))\otimes \sQ_{d_s}
\end{eqnarray*}
where the last line follows by the projection formula as $\sQ_{d_s}$ is defined over $F(d_s,V)$.
Note that
$$R(p_{s-1})_*(\Sigma^{(2)}(\sW_{d_{s-1}})\otimes\bigwedge^{r-1}(\sW_{d_s}/\sW_{d_{s-1}}))
=R(\pi_{s-1})_*(\sO_{\flag(\sW_{d_s})}(\chi))$$
where $\pi_{s-1}:\flag(\sW_{d_s})\to F(d_s,V)$ is the relative full flag
and 
$$\chi=(0,\dots,0,-2,0,-1,\dots,-1)\in \ZZ^{n-1}$$ has a -2 in the $d_{s-1}$ spot, a 0 in the $d_{s-1}+1$ spot and $(r-1)$ $-1$'s in the remaining positions.  A similar calculation using (\ref{bbw}) shows
that the above bundle is $\bigwedge^{r+1}(\sW_{d_s})[1]$.
Putting this together with the above shows that 
$$H^*(P,\sE)=R(p_s)_*(\bigwedge^{r+1}(\sW_{d_s})[1]\otimes \sQ_{d_s}).$$
For the structure morphism $\pi_s$ of $\flag(V)$,
we see that 
$$R(p_s)_*(\bigwedge^{r+1}(\sW_{d_s})[1]\otimes \sQ_{d_s})=R(\pi_s)_*(\sO(\chi'))$$
where $\chi'=(0,\dots,0,-1,\dots,-1)\in \ZZ^n$ has a string of $r+2=d_2+1$
(-1)'s at the end.  Since this weight is dominant, an application of (\ref{bbw}) produces
$\bigwedge^{r+1}(V)[1]$. Hence
$$\Ext^1_P(\sigma^*(\sF),\sG)=\bigwedge^{r+2}(V)\ne 0$$
and all other Ext groups vanish.
Note that $r+2=d_2+1\le n$ by assumption so that $\bigwedge^{r+2}(V)\ne 0$.

\noindent\textbf{Case 3: $d_1=1$, $d_2=2$.}

Note that $n\ge 3$, and we have $d_s=n-1$ and $d_{s-1}=n-2$ by the 
symmetry assumption.
Take $\sF=\sW_1$ and $\sG=\sW_{n-2}\otimes \sW_{n-1}$.
Note that 
$$\sF=\Sigma^{\alpha_1}(\sW_{d_1})\otimes \cdots \otimes\Sigma^{\alpha_s}(\sW_{d_s})$$
$$\sG=\Sigma^{\beta_1}(\sW_{d_1})\otimes \cdots \otimes \Sigma^{\beta_s}(\sW_{d_s})$$
where $\alpha_1=(1)$ and $\alpha_i=(0)$ for all $i\ne 1$
and $\beta_{s-1}=\beta_s=(1)$, $\beta_i=(0)$ for all $i\ne s-1,s$.
Since $d_i\ge 1$ for all $i$, these
bundles are clearly in the exceptional collection constructed in (\ref{t:kapranov}).
Then 
$$\Ext^*_P(\sigma^*(\sF),\sG)=H^*(P,\sE)$$
where $\sE=\sQ_{n-1}\otimes \sW_{n-2}\otimes \sW_{n-1}$.
We factor the structure morphism of $F=F(d_1,\dots,d_s,V)$
into $q:F(d_1,\dots,d_s,V)\to F(d_{s-1},d_s,V)$ and the structure morphism $t$ for 
$F(d_{s-1},d_s,V)$.
The same calculation as in the second case shows that for $P=F(d_1,\dots,d_s,V)$
we have 
$$H^*(P,\sE)=H^*(F(d_{s-1},d_s,V),\sE)$$
since our bundle $\sE$
is defined over $F(d_{s-1},d_s,V)=F(n-2,n-1,V)$.
We now factor the structure morphism $t$ for $F(n-2,n-1,V)$
into the relative Grassmann bundle $p_{s-1}:F(n-2,n-1,V)\to F(n-1,V)=\gr(n-1,V)$ and the structure morphism $p_s$ for $\gr(n-1,V)$. 
So 
$$H^*(F(n-2,n-1,V),\sE)=R(p_s)_*(R(p_{s-1})_*(\sE))$$
We now analyse $R(p_{s-1})_*(\sE)$:
Note that we have an exact sequence
$$0\to \sW_{n-1}/\sW_{n-2}\to \sQ_{n-2}\to \sQ_{n-1}\to 0$$
Tensoring with $\sW_{n-2}\otimes \sW_{n-1}$ we get an exact
sequence
$$0\to \sE'\to \sE\to \sE''\to 0$$
where $\sE'=(\sW_{n-1}/\sW_{n-2})\otimes\sW_{n-2}\otimes \sW_{n-1}$
and $\sE''=\sQ_{n-1}\otimes \sW_{n-2}\otimes \sW_{n-1}$. 
We then have an exact sequence
$$0\to R(p_{s-1})_*(\sE')\to R(p_{s-1})_*(\sE)\to R(p_{s-1})_*(\sE'')\to 0$$
in the derived category.
We first analyse $R(p_{s-1})_*(\sE'')$.
By the projection formula, we have
$$R(p_{s-1})_*(\sE'')=R(p_{s-1})_*(\sW_{n-2})\otimes \sW_{n-1}\otimes \sQ_{n-1}$$
But 
$$R(p_{s-1})_*(\sW_{n-2})=R(\pi_{s-1})_*(\sO_{\flag(\sW_{n-1})}(\chi))$$
where $\pi_{s-1}:\flag(\sW_{n-1})\to F(n-1,V)$ is the relative full flag
and 
$$\chi=(0,\dots,0,-1,0)\in \ZZ^{n-1}.$$  The last two entries 
of $\chi+\rho$ are 1 and so by (\ref{bbw}), we see that 
$R(p_{s-1})_*(\sW_{n-2})=0$ and so $R(p_{s-1})_*(\sE'')=0$.
Then
$$R(p_{s-1})_*(\sE)=R(p_{s-1})_*(\sE')=R(p_{s-1})_*((\sW_{n-1}/\sW_{n-2})\otimes \sW_{n-2})\otimes \sW_{n-1}$$
where the last equality follows from the projection formula.
Then 
$$R(p_{s-1})_*(\sW_{n-2}\otimes (\sW_{n-1}/\sW_{n-2}))=R(\pi_{s-1})_*(\sO_{\flag(\sW_{n-1})}(\chi'))$$
where $\pi_{s-1}:\flag(\sW_{n-1})\to F(n-1,V)$ is the relative full flag
and $$\chi=(0,\dots,0,-1,-1)\in \ZZ^{n-1}.$$  Since $\chi$ is dominant, so 
is $\chi+\rho$, and so  by (\ref{bbw}), we see that 
$$R(p_{s-1})_*(\sW_{n-2}\otimes \sW_{n-1}/\sW_{n-2})=\bigwedge^2(\sW_{n-1}).$$
Then by Littlewood Richardson, we have
$$R(p_{s-1})_*(\sE')=\bigwedge^2(\sW_{n-1})\otimes \sW_{n-1}=\bigwedge^3(\sW_{n-1})\oplus \Sigma^{(2,1)}(\sW_{n-1})$$
[Note that here we have $n\ge 3$ and so $\Sigma^{(2,1)}(\sW_{n-1})\ne 0$ but 
$\bigwedge^3(\sW_{n-1})\ne 0$ if and only if $n\ge 4$.
This will turn out not to matter as this term vanishes in the next step.]
So 
$$R(p_s)_*(R(p_{s-1})_*(\sE))=R(p_s)_*(\bigwedge^3(\sW_{n-1}))\oplus 
R(p_s)_*(\Sigma^{(2,1)}(\sW_{n-1}))$$
Let $\pi_s$ be the structure morphism for $\flag(V)$.
Then we have
$$R(p_s)_*(\bigwedge^3(\sW_{n-1}))=R(\pi_s)_*(\chi_1)$$
where $\chi_1=(0,\dots,0,-1,-1,-1,0)\in \ZZ^n$.
Here $\chi_1+\rho$ has a repeat of 1 in the last two entries and 
so $R(p_s)_*(\bigwedge^3(\sW_{n-1}))=0$ by (\ref{bbw}).
But we also have
$$R(p_s)_*(\Sigma^{(2,1)}(\sW_{n-1}))=R(\pi_s)_*(\chi_2)$$
where $\chi_2=(0,\dots,0,0,-1,-2,0)\in \ZZ^n$.
Here $\chi_2+\rho=(n,n-1,\dots,4,2,0,1)$. Letting $w=(n-1,n)$,
we see that $\alpha_2=w\cdot \chi_2=(0,\dots,0,-1,-1,-1)$
and 
so $R(p_s)_*(\Sigma^{(2,1)}(\sW_{n-1}))=\bigwedge^3(V)[1]$ by (\ref{bbw}).
Then we have
$$\Ext^1_P(\sigma_*(\sF),\sG)=\bigwedge^3(V)\ne 0$$
and all other Ext groups vanish.
\end{proof}

\section{Applications to arithmetic toric varieties}

\begin{theorem}(\cite{Be})\label{t:belinson}
The derived category $D^b(\mathbb{P}^n)$ is generated by the strong exceptional collection
$$ \{ \mathcal{O}(-n), \mathcal{O}(-n+1), \cdots, \mathcal{O}(-1), \mathcal{O} \}. $$
\end{theorem}

Now let us fix some notation. For projective space $\mathbb{P}^n$, we always choose $\{\mathcal{O}(1)\}$ as a basis of $\Pic(\mathbb{P}^n)=\mathbb{Z}$; for a projective bundle $p: \mathbb{P}(\mathcal{E}) \to \mathbb{P}^n$, we always choose $\{p^*\mathcal{O}(1), \mathcal{O}_{\mathcal{E}}(1)\}$ as a basis of $\Pic(\mathcal{P}(\mathcal{E}))=\mathbb{Z} \oplus \mathbb{Z}$ and we denote by $\mathcal{O}(i,j)$ the tensor product $p^*\mathcal{O}(i) \otimes \mathcal{O}_{\mathcal{E}}(j)$; and so on.

\begin{proposition}
Consider projective bundle $p: \mathbb{P}(\mathcal{E}) \to \mathbb{P}^n$. Assume $\mathcal{E} = \mathcal{L}_1 \oplus \cdots \oplus \mathcal{L}_{r+1}$ such that $\mathcal{L}_i$ is in $\Pic(\mathbb{P}^n)^+ \simeq \mathbb{Z}_{\geq 0}$ for all $1\leq i \leq r+1$. Then 
$$ (\mathcal{O}(-n,-r), \mathcal{O}(-n+1,-r), \cdots, \mathcal{O}(0,-r), \cdots, \mathcal{O}(-n,0), \cdots, \mathcal{O}(0,0))$$
is a full strong exceptional collection of coherent sheaves on $\mathbb{P}(\mathcal{E})$.
\end{proposition}

\begin{proof}
By Theorem \mref{t:belinson} and \cite[Corollary 2.7]{Or}, we only need to show this set is a strong set, which follows from an easy computation of  projective space cohomology.
\end{proof}

More generally, we  have
\begin{corollary}\label{c:FSEcor}
Consider a series of projective bundles $\mathbb{P}(\mathcal{E}_m) \to \cdots \to \mathbb{P}(\mathcal{E}_i) \to \mathbb{P}^{r_0}$. Assume $\mathcal{E}_i$ is decomposable of rank $r_i+1$ and all its summands are in $\Pic(\mathbb{P}(\mathcal{E}_{i-1})^+ \simeq (\mathbb{Z}_{\geq 0})^{\oplus i}$ for all $1\leq i \leq m$. Then the set
$$ \{ \mathcal{O}(j_0, j_1, \cdots, j_m): -r_i \leq j_i \leq 0, 0 \leq i \leq m \}$$
is a full strong exceptional collection of coherent sheaves ono $\mathbb{P}(\mathcal{E}_m)$ by the lexicographical order on $(j_0, j_1, \cdots, j_m)$.
\end{corollary}

Recall that, see \cite{ELST}, \emph{an arithmetic torus} over $k$ of rank n is an algebraic group $\mathcal{T}$ over $k$ such that $\mathcal{T}_l \simeq T_{N, l}$ for some finite Galois extension $l/k$ and lattice $N$ of rank $n$, and an \emph{arithmetic toric variety} over $k$ is a pair $(Y, \mathcal{T})$, where $\mathcal{T}$ is an arithmetic torus over $k$ and $Y$ is a normal variety over $k$ equipped with a faithful action of $\mathcal{T}$ which has a dense orbit. Let $(Y_{\Sigma,l}, T_{N,l})$ be its split toric variety and $G=\Gal(l/k)$, then the $G$-action on $(Y_{\Sigma,l}, T_{N,l})$ is determined by a conjugacy class of group homomorphisms $\varphi: G \to \Aut(N)$ such that $\varphi(G) \subseteq \Aut_{\Sigma}$.

\begin{definition} (\cite{B})
Let $\Sigma$ be a smooth complete fan, we call a nonempty subset $\mathcal{P} = \{x_1, \cdots, x_k\} \subseteq \Sigma(1)$ a \emph{primitive collection} if for each element $x_i \in \mathcal{P}$, the remaining elements $\mathcal{P}\setminus \{x_i\}$ generate a $(k-1)$-dimensional cone in $\Sigma$, while $\mathcal{P}$ itself does not generate any $k$-dimensional cone in $\Sigma$. We will call $\Sigma$ a \emph{splitting fan} if any two different primitive collections in $\Sigma(1)$ are disjoint.
\end{definition}

\begin{theorem}
{Let $(X, \mathcal{T})$ be an arithmetic toric variety over $k$, whose split toric variety corresponding to a splitting fan, then there exists a tilting bundle on $X$.
}
\end{theorem}

\begin{proof}
Let $X_l$ be the corresponding split toric variety with splitting fan $\Sigma$ in a lattice $N$, where $l/k$ is a Galois extension with Galois group $G$. By \cite[Theorem 4.3]{B}, we have a projectivization  $X_l=\mathbb{P}(\mathcal{E}) \to X'_l$, which corresponds to a primitive collection $\mathcal{P}=\{x_1, x_2, \cdots, x_{k+1}\} \subseteq \Sigma(1)$ with primitive relation $x_1 + x_2 + \cdots + x_{k+1} = 0$ by \cite[Proposition 4.1]{B}. Since $\Sigma(1)$ generates $\Sigma$, the action of $G$ on $\Sigma$ is determined by the action of $G$ on $\Sigma(1)$. As $G$ preserves the primitive relationship and $\mathcal{P}$ has no intersection with any other primitive collection in $\Sigma(1)$, we must have either $g(\mathcal{P})=\mathcal{P}$ or $g(\mathcal{P}) \cap \mathcal{P} = \emptyset$ for any $g \in G$. Let the distinguished primitive collections $\mathcal{P}_1, \cdots, \mathcal{P}_m$ be the images of $\mathcal{P}$ under the action of $G$. Again, by \cite[Proposition 4.1]{B}, these primitive collections determine a series of projective bundles $\mathbb{P}(\mathcal{E}_1) \to \cdots \to \mathbb{P}(\mathcal{E}_m) \to Y_l$, where $Y_l$ is also a toric variety with splitting fan by \cite[Theorem 4.3]{B}.

By \cite[page 59]{O}, we may construct the fan $\Sigma$ from the fan $\Sigma_{Y_l}$. The Galois $G$-action on $X_l$ induces an Galois $G$-action on $Y_l$. Let $Y_l$ descend to $(Y_k, \mathcal{T}')$. Then we have a compatible commutative diagram:
\[
\xymatrix{  X_l \ar[r]
             \ar[d] & (X, \mathcal{T}) \ar[d] \\
            Y_l \ar[r]  & (Y_k, \mathcal{T}').  }
\]

Actually, for every $\tau \in S_m$, the permutation group of the set $\{1, 2, \cdots, m\}$, we have a series of projective bundles $\mathbb{P}(\mathcal{E}^{\tau}_{\tau_1}) \to \mathbb{P}(\mathcal{E}^{\tau}_{\tau_2}) \to \cdots \to \mathbb{P}(\mathcal{E}^{\tau}_{\tau_m}) \to Y_l$. Thus each of these primitive collections $\mathcal{P}_1, \cdots, \mathcal{P}_m$ induces a projective bundle $\mathbb{P}(\mathcal{E}_i) \to Y_l, i=1,\cdots,m$. The $G$-action on $X_l$ induces commutative diagrams
\[
\xymatrix{  \mathbb{P}(\mathcal{E}_i) \ar[r]^-{\rho_{g_{i,j}}}
             \ar[d] & \mathbb{P}(\mathcal{E}_j) \ar[d] \\
            Y_l \ar[r]^-{\rho_{g_{i,j}}}  & Y_l  }
\]
for $1 \leq i, j \leq m$. So we may assume that $\{\rho_g^*(\mathcal{E}): g \in G\} = \{\mathcal{E}_1, \cdots, \mathcal{E}_m\}$.

Denote by $X'_l = \mathbb{P}(\mathcal{E}_1) \times_{Y_l} \cdots \times_{Y_l} \mathbb{P}(\mathcal{E}_m)$, we can see that $\Sigma_{X_l} \simeq \Sigma_{X'_l}$, and hence $X_l \simeq \mathbb{P}(\mathcal{E}_1) \times_{Y_l} \cdots \times_{Y_l} \mathbb{P}(\mathcal{E}_m)$.

 Thus we get the following compatible commutative diagram
\[
\xymatrix{  X_l= \mathbb{P}(\mathcal{E}_1) \times_{Y_l} \cdots \times_{Y_l} \mathbb{P}(\mathcal{E}_m) \ar[r]
             \ar[d] & (X, \mathcal{T}) \ar[d] \\
            Y_l \ar[r]  & (Y_l, \mathcal{T}'),  }
\]
where $\mathcal{E}_i$ is a decomposable bundle for all $1\leq i \leq m$ by \cite[Lemma 1.1]{DS}.

Iteratively, we get the following diagram:
\[
\xymatrix{  X_{t, l}= \mathbb{P}(\mathcal{E}_{t, 1}) \times_{X_{t-1,l}} \cdots  \times_{X_{t-1,l}} \mathbb{P}(\mathcal{E}_{t, m_t}) \ar[r]
             \ar[d] & (X_{t,k}, \mathcal{T}_l)= (X, \mathcal{T}) \ar[d] \\
             \vdots
             \ar[d] & \vdots \ar[d] \\
             X_{2,l}= \mathbb{P}(\mathcal{E}_{2, 1}) \times_{X_{1,l}} \cdots \times_{X_{1,l}} \mathbb{P}(\mathcal{E}_{2, m_2}) \ar[r]
             \ar[d] & (X_{2,k}, \mathcal{T}_2) \ar[d] \\
             X_{1,l}= \times_l^{m_1}  \mathbb{P}(\mathcal{E}_1) \ar[r]
             \ar[d] & (X_{1,k}, \mathcal{T}_1) \ar[d] \\
             X_{0,l}= \spec\,l  \ar[r] & \spec\,k }
\]
where $\mathcal{E}_1$ is a decomposable vector bundle of rank $r_1 + 1$ over $X_{0,L}$ and $\mathcal{E}_{i, j_i}$ is a decomposable vector bundle of rank $r_i+1$ over $X_{i-1,l}$ and $\{\rho_g^*(\mathcal{E}_{i, 1}): g \in G\} = \{\mathcal{E}_{i, 1}, \cdots, \mathcal{E}_{i, m_i}\}$ for $2 \leq i \leq t$ and $1 \leq j_i \leq m_i$.

As we know that $\Pic(X_{i,l}) \simeq \mathbb{Z}^{\oplus (m_1 + \cdots + m_i)}$, we may assume all the line bundle summands of $\mathcal{E}_i$ are in $(\mathbb{Z}_{\geq 0})^{\oplus (m_1 + \cdots + m_{i-1})}$ for all $2 \leq i \leq t$.

Without causing confusion, we use the same notation $\mathcal{O}_{\mathbb{P}(\mathcal{E}_{i,j})}(s) \,(1 \leq i \leq t,\, 1 \leq j \leq m_i)$ to denote the corresponding component in $\Pic(X_{h,L})$ for all $i \leq h \leq t$.  Denote by
\begin{eqnarray*}
&& \mathcal{O}(j_{1,1}, \cdots, j_{1,m_1}, \cdots, j_{t,1}, \cdots, j_{t,m_t}) \\
&=& (\mathcal{O}_{\mathbb{P}(\mathcal{E}_1)}(j_{1,1}), \cdots, \mathcal{O}_{\mathbb{P}(\mathcal{E}_1)}(j_{1,m_1}), \cdots, \mathcal{O}_{\mathbb{P}(\mathcal{E}_{t,1})}(j_{t,1}), \cdots, \mathcal{O}_{\mathbb{P}(\mathcal{E}_{t,m_t})}(j_{t,m_t})),
\end{eqnarray*}
where $-r_i \leq j_{i,k_i} \leq 0$ and $1 \leq k_i \leq m_i$ for $1 \leq i \leq t$.

Then by \mref{c:FSEcor}, the set $$\{\mathcal{O}(j_{1,1}, \cdots, j_{1,m_1}, \cdots, j_{t,1}, \cdots, j_{t,m_t}): -r_i \leq j_{i,k_i} \leq 0, 1 \leq k_i \leq m_i\}$$ is a full strong exceptional collection of $D^b(X_l)$ by the lexicographical order on $(j_{1,1}, \cdots, j_{1,m_1}, \cdots, j_{l,1}, \cdots, j_{l,m_l})$. For any $g \in G$, we have
\begin{eqnarray*}
&& \rho_g^*\mathcal{O}(j_{1,1}, \cdots, j_{1,m_1}, \cdots, j_{t,1}, \cdots, j_{t,m_t}) \\
&=& \mathcal{O}(j_{1,\tau_{1,g}(1)}, \cdots, j_{1,\tau_{1,g}(m_1)}, \cdots, j_{t,\tau_{t,g}(1)}, \cdots, j_{t,\tau_{t,g}(m_t)}),
\end{eqnarray*}
where $\tau_{i,g}, 1 \leq i \leq t$, are permutations of the corresponding sets $\{1, \cdots, m_i\}$. So it is also in the same set.

Let $$\mathcal{T} = \oplus \rho_g^*\mathcal{O}(j_{1,1}, \cdots, j_{1,m_1}, \cdots, j_{t,1}, \cdots, j_{t,m_t}),$$ then $\mathcal{T}$ is a tilting sheaf on $X_{t,l}$ by Lemma \mref{l:FSEcollection}, and $\mathcal{T}$ descends to a tilting bundle on $X$ by \mref{p:main}.
\end{proof}


\end{document}